\documentclass[12pt]{article}
\usepackage{mathrsfs}
\usepackage{graphicx}
\usepackage{amsmath}
\usepackage{amssymb}
\usepackage{amsthm}

\usepackage{thmtools}
\usepackage{amsfonts}

\usepackage{bbding}
\usepackage{CJK,CJKnumb}
\usepackage{comment}
\usepackage{mathtools}
\usepackage{xcolor}
\usepackage{authblk}

\usepackage{etoolbox}
\patchcmd{\thebibliography}{\chapter*}{\section*}{}{}
\usepackage{ltablex}

\usepackage[nospace,noadjust]{cite}
\usepackage{calc}
\usepackage{geometry}
\usepackage{cite}
\usepackage{tikz-cd}
\usepackage{tikz}
\usetikzlibrary{matrix}

\usepackage{indentfirst}
\usepackage{longtable}

\allowdisplaybreaks

\numberwithin{equation}{section} 

\newtheorem{definition}{Definition}[section]

\newtheorem{theorem}[definition]{Theorem}

\newtheorem{proposition}[definition]{Proposition}
\newtheorem{corollary}[definition]{Corollary}
\newtheorem{lemma}[definition]{Lemma}
\newtheorem{conjecture}[definition]{Conjecture}

\newtheorem{example}[definition]{Example}
\newtheorem*{remark}{Remark}

\newcommand{\Gal}{\mathrm{Gal}}

\title {A Gap Principle for Subvarieties with Finitely Many Periodic Points}
\author[1]
{
Keping Huang
\thanks{keping.huang@rochester.edu}
}

\affil[1]{Department of Mathematics, University of Rochester}
\date{\vspace{-4em}}

\begin{document}
\maketitle

\def\Z{{\bf Z}}

\begin{abstract}
Let $f:X\rightarrow X$ be a quasi-finite endomorphism of an algebraic variety $X$ defined over a number field $K$ and fix an initial point $a\in X$. 
We consider a special case of the dynamical Mordell-Lang Conjecture, 
where the subvariety $V$ contains only finitely many periodic points and does not contain any positive-dimensional periodic subvariety. 
We show that the set $\{n\in \mathbb{N}~|~f^n(a) \in V \}$ satisfies a strong gap principle. 
\end{abstract}

\section{Introduction}

The dynamical Mordell-Lang Conjecture predicts that given an endomorphism $f: X\rightarrow X$ of a complex quasi-projective variety $X$,
for any point $a\in X$ and any subvariety $V\subsetneq X$,
the set $S_V:=\{n\in\mathbb{N}_{\ge 0}~|~f^{n}(a) \in V\}$ is a finite union of arithmetic progressions (sets of the form
$\{a, a+d, a+ 2d,\dots \} $ with $a,d\in \mathbb{N}_{\ge 0})$.
The dynamical Mordell-Lang Conjecture was proposed in \cite{GT09}.
See also \cite{Bel06} and \cite{Den92} for earlier works.
In the case of \'etale maps we know that the dynamical Mordell-Lang Conjecture is true
(see \cite{bgt10} and \cite{BGT16}). 
Xie proved in \cite{Xie15} the dynamical Mordell-Lang Conjecture for polynomial endomorphisms of the affine plane.

The $p$-adic interpolation method is one of the most important tools for tackling the dynamical Mordell-Lang Conjecture. The basic idea is to construct a $p$-adic analytic function $G: \mathbb{Z}_p \rightarrow X$ such that $G(n) = f^n(a)$. This allows one to use the tools from $p$-analytic function to calculate the set $S_V$. 

However, in general, it's unknown whether such an interpolation exists for any endomorphism $f$ or any initial point $a$. For some evidence of its nonexistence see \cite{BGHKTT13}. Then we might not be able to prove the dynamical Mordell-Lang Conjecture using the $p$-adic interpolation in those cases. However, by a discovery of Bell, Ghioca and Tucker (see \cite{BGT16}, Theorem 11.11.3.1 or Theorem \ref{Bell} below), we might expect an approximating function $G$ such that $G(n)$ approximates $f^n(a)$ very closely. 
Suppose $Q$ is a polynomial vanishing on $V$, 
then the roots $Q\circ G$ give much restriction to those $n$ 
such that $f^n(a)\in V$. 
This allows us to prove a weaker version of the dynamical Mordell-Lang Conjecture in certain cases.

To apply the approximation result, we need to exclude a case where the orbit converges $p$-adically to a periodic point on $V$. One way to guarantee this is to ensure that the residue class of the orbit not meet the periodic point mod $p$ after a certain number of iterates. 
Theorem 1 of \cite{BGHKTT13} gives this kind of avoidance result for 
$\mathbb{P}^1$. 
In this paper, we generalize this result to quasi-finite endomorphisms of quasi-projective varieties. 
This enables us to prove a strong gap principle under certain conditions. 

The main theorem of this paper is the following theorem. For the definition of nonsingular variety, see Section 1.5 of \cite{Har77}. 

\begin{theorem}\label{main}
  Suppose $f:X\rightarrow X$ is a quasi-finite endomorphism of a nonsingular variety $X$. Let $V\subseteq X$ be a subvariety. 
  Assume that $X, V,f$ are all defined over a number field $K$. 
  Assume that there are only finitely many periodic points in $V(\overline{K})$ and that $V$ does not contain a positive-dimensional periodic subvariety. 
  Assume that $a \in X(K)$ is not preperiodic under $f$, then
the set $S_V:= \{ n\in \mathbb{N}_{\ge 0}~|~f^n(a) \in V(K) \}$
has the property that 
$$\#\{i\le n:~i\in S_V\} = o(\log^{(m)}(n))$$
for any fixed $m\in \mathbb{N}_{>0}$ (where $\log^{(m)}$ is the $m$-th iterated logarithmic function).  
\end{theorem}

This theorem is an analog of Theorem 1.4 of \cite{BGKT10}. 
In a broader sense, both theorems can be thought of as dynamical analogs of the 
gap principle in Diophantine geometry. See, for example, \cite{DR55} and \cite{Mum65} for the Diophantine case. 
If $C$ is a curve of genus greater than 1, 
Mumford showed that if we order the rational points of $C$ according to Weil height, then 
their Weil height grows exponentially. 
Faltings later proved in \cite{Fal83} that the number of rational points of $C$ is actually finite. 
Here we prove the density of $S_V$ by proving a gap between consecutive numbers in certain arithmetic progressions. 
One difference is that in both results in dynamics the ``gaps" are much larger than in the Diophantine geometry. 
But our ultimate goal is also to show that $S_V$ is finite.

We say that a morphism $f: X \rightarrow X$ of a normal
variety $X$ is {\em polarizable} if there is an ample divisor $L$ of $X$
such that $f^* L \cong L^{\otimes d}$ for some $d > 1$. 
Our result is related with the following. 
\begin{conjecture}[Dynamical Manin-Mumford, \cite{Zha95}]\label{dmm}
Suppose $f: X\rightarrow X$ is a polarizable endomorphism of a variety $X$. 
If a subvariety $V\subset X$ contains a Zariski dense set of preperiodic points, then $V$ is itself preperiodic. 
\end{conjecture}

In the case when $f$ is split rational map on $(\mathbb{P}^1)^n$, the dynamical Manin-Mumford Conjecture is proved in \cite{GNY17}. 
Unfortunately, there is a counterexample in \cite{GTZ11} to Conjecture \ref{dmm}. 
In the same paper, the authors also proposed a more refined conjecture.

If the dynamical Manin-Mumford Conjecture is true for $f:X\rightarrow X$ and $V\subset X$, then the condition that $V$ does not contain a positive-dimensional subvareity would imply that the set of preperiodic points on $V$ is not Zariski dense. 
In particular, if $V$ is a curve, then the condition that $V$ is not periodic would imply that there are finitely many periodic points on $V$. 
For more about the dynamical Manin-Mumford Conjecture, see \cite{Zha06}. 
On the other hand, \cite{Fak03} shows that if $f$ is polarizable, 
then the subset of $X(\overline{K})$ consisting of periodic
points of $f$ is Zariski dense in $X$. 

By Theorem \ref{main} and the above discussion, we have the following 

\begin{theorem}\label{unconditional}
  Let $X=(\mathbb{P}^1)^n$. 
  Suppose $f: X \rightarrow X$ is given by
  $(x_1,\dots,x_n)\mapsto (f_1(x_1),\dots, f_n(x_n))$ where each $f_i$ is one-variable rational function of degree at least $2$. 
  Let $V\subseteq X$ be a non-periodic curve. 
  Suppose $X, V,f$ are all defined over a number field $K$. 
Then the set $S_V:= \{ n\in \mathbb{N}_{\ge 0}~|~f^n(a) \in V(K) \}$
has the property that 
$$\#\{i\le n:~i\in S_V\} = o(\log^{(m)}(n))$$
for any fixed $m\in \mathbb{N}_{>0}$ (where $\log^{(m)}$ is the $m$-th iterated logarithmic function).  
\end{theorem}

The avoidance result below is a key step towards the proof of Theorem \ref{main}. 
It is useful in many other contexts. 
It is a generalization of Theorem 5 of \cite{BGHKTT13} and its
proof goes parallel with that of the latter theorem, with simplifications in certain steps. 

\begin{theorem}\label{key}
  Suppose $f:X\rightarrow X$ is a quasi-finite endomorphism of a nonsingular variety $X$. Let $V\subseteq X$ be a subvariety. 
  Assume that $X, V,f$ are all defined over a number field $K$. 
  Assume that there are only finitely many periodic points $\gamma_1,\dots, \gamma_r$ in $V(\overline{K})$ and that $V$ does not contain a positive-dimensional periodic subvariety. 
  Then there is a finite set of places $S$ of $K$, a positive integer $M$ and a set $\mathcal{P}$ of primes of $K$ with positive density such that for all $\mathfrak{p}'\in \mathcal{P}$, for all $m\ge M$, for all $1 \le i \le r$, and for all $a\in X(\mathcal{O}_{S,K})$ not preperiodic, we have for the reductions 
$f_\mathfrak{p'} ^m(a_\mathfrak{p'})\neq (\gamma_i)_{\mathfrak{p}'}$. 
\end{theorem}

We follow the idea of the proof in \cite{BGHKTT13} to prove Theorem \ref{key}. 
First, we find a single prime $\mathfrak{p}$ at which the reduction of $a$ does not hit $\gamma_i$ after the $M$-th iteration. 
Note that $\mathfrak{p}$ might ramify, but in contrast with the treatment in \cite{BGHKTT13} we use Corollary \ref{unified} to obtain a unified way to find a Frobenius ``coset'' in either ramified case or unramified case, and then we use the Chebotarev density theorem to find a family of primes with the same Frobenius conjugacy class, and hence prove the avoidance theorem.

For the proof of Theorem \ref{main}, after some reductions we use Theorem \ref{Bell} to find an approximate interpolating function $G$. This allows us to use the zeros of $G$ to obtain information about the set $S_V$. We need Theorem \ref{key} to avoid one case in which we cannnot say much about $S_V$. A key technical result is Proposition \ref{book}, 
which is a modification of Theorem 11.11.3.1 of \cite{BGT16}.

The Organization of this paper is as follows. 
Section 2 gives the basic notations and definitions, as well as some simple reductions. 
In Section 3, we prove the avoidance result. 
In Section 4, we apply the approximate $p$-adic interpolation to prove our main theorem.

\section{Definitions and Preliminaries}\label{notation}

\begin{definition}
Suppose $\phi:X\rightarrow Y$ is a morphism of algebraic varieties. We say that $\phi$ is quasi-finite if each point $y\in Y$ has finitely many preimages. 
\end{definition}

We need the following simple observation.

\begin{lemma} \label{iter}
If Theorem \ref{main} holds for $f^k$ with initial points $a, f(a),f^2(a),\dots,f ^{k-1}(a)$.
Then Theorem \ref{main} holds for $f$ with initial point $a$. 
\end{lemma}

\begin{proof}
The number $k$ is independent of $m$ and $n$. 
Then the result is clear as a sum of $k$ numbers of size $o(\log^{(m)}(n))$ is still of size $o(\log^{(m)}(n))$. 
\end{proof}

Suppose $f:X \rightarrow X$ is a quasi-finite endomorphism of a variety $X$. Let $V \subset X$
be a subvariety. Assume that $X, V, f$ are all defined over a number field $K$. Suppose
$\gamma_1,\dots, \gamma_r$ are all the periodic points of $f$ on $V$. Replacing $f$ by an iterate we may
assume that all $\gamma_i$ are
fixed points. 
Note that by Lemma \ref{iter} this does not affect the gap property in Theorem \ref{main}. 
Extending $K$ if necessary we may assume that all $\gamma_i$ are defined over $K$. 
Suppose $S$ is a finite set containing all the infinity places of $K$ and all the finite places where the reduction of $X$ is not a nonsingular variety or the reduction of $f$ is not well-defined. 
Choose an $S$-integral model $\mathcal{X}$ of $X$, that is, a scheme $\mathcal{X}$ flat over the ring $\mathcal{O}_{S,K}$ of $S$-integers in $K$, 
such that the generic fiber $\mathcal{X} \times_{\mathcal{O}_{S,K}} K$ is isomorphic to $X$.


At the expense of expanding the size of $S$ by a finite number, we may do a couple of simplifications. 
We may assume that $f$ can be extended to a morphism of schemes $f: \mathcal{X} \rightarrow \mathcal{X}$.
We may also assume that the points $\gamma_1,...,\gamma_r$ are defined over $\mathcal{O}_{S,K}$. Also denote them by $\gamma_1,...,\gamma_r$. 
Furthermore assume
$f^{-1}(\gamma_i)$ do not hit $\gamma_i$ modulo $\mathfrak{p}$ for all $\mathfrak{p} \in S$ and for all $1\le i \le r$. 
For a prime $\mathfrak{p}$ of $K$ outside $S$, denote by $(\gamma_1)_{\mathfrak{p}},...,(\gamma_r)_{\mathfrak{p}}, $ the extended points of $\gamma_1,...,\gamma_r$ on the special fiber ${X}_{\mathfrak{p}}:=\kappa(\mathfrak{p})\times_{\mathcal{O}_{S,K}} \mathcal{X}$, 
and denote by $f_\mathfrak{p}$ the reduction of $f$ at $\mathfrak{p}$.

For any prime $\mathfrak{p} $ of $K$ such that the reduction ${X}_{\mathfrak{p}}$ is nonsingular, we use the notations as below. 
Let $p$ be the prime number divisible by $\mathfrak{p}$. 
Let $\mathfrak{o}_v$ be the completion of the local ring $(O_{K})_\mathfrak{p}$ 
and suppose $\pi$ is a uniformization element of $\mathfrak{o}_v$. 
For $f \in \mathfrak{o}_v[x_1,\dots,x_m]$ we let $\|f \|$ be the supremum of the absolute values of the coefficients of $f$. 
We let $\mathfrak{o}_v\langle x_1,\dots,x_m\rangle$ be the completion of $\mathfrak{o}_v[x_1,\dots,x_m]$ with respect to $\| \cdot \|$; 
this is
called the Tate algebra and it consists of all power series in $x_1,\dots,x_m$ with the property
that the absolute values of its coefficients tends to $0$.

\section{The Proof of Theorem \ref{key}}

We follow the idea of the proof in \cite{BGHKTT13}.
We need some lemmas.

\begin{lemma}\label{nonper}
With $f, X, \mathcal{X},S$ 
as in Section \ref{notation}, Fix a prime $\overline{\mathfrak{p}}$ of $\mathcal{O}_{\bar{K}} $ such that $\overline{\mathfrak{p}} \cap \mathcal{O}_{S,K} \notin S$. 
Assume that $\gamma_{\mathfrak{p}}$ is not a periodic point modulo $\mathfrak{p}$. Then for all finite extension $L/K$, there is an integer $M$ such that for all $m \ge M$ and for all $\beta \in X(\overline{K})$ with $f^m(\beta)=\gamma$, we have $[\mathcal{O}_{L(\beta)}/\tau: \mathcal{O}_{L}/\mathfrak{q}]>1$ where $\tau =\overline{\mathfrak{p}} \cap \mathcal{O}_{L(\beta)}$ and $\mathfrak{q}=\overline{\mathfrak{p}}\cap \mathcal{O}_L$.
\end{lemma}

\begin{proof}
Since $\gamma$ is not periodic modulo $\mathfrak{p}$ and the set $X_\mathfrak{p}(\mathcal{O}_{L}/\mathfrak{q})$ is finite, we know that there exists an $M $ such that for all $m\ge M$, the set 
$\{x\in X_{\mathfrak{p}} (\mathcal{O}_L/\mathfrak{q})~|~f_{\mathfrak{p}}^m(x)\neq \gamma_{\mathfrak{p}}\}$ 
is empty. It follows that $[\mathcal{O}_{L(\beta)}/\tau: \mathcal{O}_{L}/\mathfrak{q}]>1$. 
\end{proof}

Assume that $\gamma$ is a fixed point modulo $\mathfrak{p}$. 
Suppose $\eta_j \in  f^{-1}(\gamma) \setminus \{\gamma\}$ are the preimages of $\gamma$ other than $\gamma$ itself. 
Let $E$ be the composiitum of $K$ with all the defining fields of $\eta_j$.

\begin{lemma}\label{fixed}
With $f, X, \mathcal{X}, S$ 
as in Section \ref{notation} and in the above paragraph. Fix a prime $\overline{\mathfrak{p}}$ of $\mathcal{O}_{\bar{K}} $ such that $\overline{\mathfrak{p}} \cap \mathcal{O}_{S,K} \notin S$. 
Then for all finite extension $L/E$, there is an integer $M$ such that for all $m \ge  M$ and for all $\beta \in  X(\overline{K})$ with $f^m(\beta) = \gamma$ and $f^t(\beta)\neq \gamma$ for $t<m$, we have $[\mathcal{O}_{L(\beta)}/\tau : \mathcal{O}_L/\mathfrak{q}]  > 1$ where $\tau =\overline{\mathfrak{p}} \cap \mathcal{O}_{L(\beta)}$ and $\mathfrak{q}=\overline{\mathfrak{p}}\cap \mathcal{O}_L$.
\end{lemma}

\begin{proof} 
This is a very slight generalization of the proof of Proposition 1 of \cite{BGHKTT13}. 
We apply Lemma \ref{nonper} to each $\eta_j$. The assumption about $S$ implies that $\eta_j \not\equiv \gamma \pmod {\overline{\mathfrak{p}}}$. 
Therefore all $(\eta_j)_{\mathfrak{p}}$ are not periodic under $f_{\mathfrak{p}}$. 
We can find $M_j $'s such that for all $m\ge M_j$, if $\beta\in  X(K)$ satisfies $f^m(\beta) = \eta_j$ and $f^t(\beta)\neq \eta_j$ for $t<m$, then we have $[\mathcal{O}_{L(\beta)}/\tau : \mathcal{O}_L/\mathfrak{q}] > 1$.
Since there are finitely many $\eta_j$, 
we can set $M:=\max M_j +1$. 
Now for all $m\ge M$, if $\beta\in  X(K)$ satisfies $f^m(\beta) = \gamma$ and $f^t(\beta)\neq \gamma$ for $t<m$, then $f^{m-1}(\beta) = \eta_j$ for some $j$. 
Hence we have $[\mathcal{O}_{L(\beta)}/\tau : \mathcal{O}_L/\mathfrak{q}] > 1$.
\end{proof}

The lemma below shows that passing to integral closures affects only finitely many primes. 
\begin{lemma}\label{integral}
 Suppose $B' \subseteq B''$ are both rings with field of fraction $E$. 
 Assume that $B''$ is integral over $B'$. 
 Then there is prime $\mathfrak{q}$ of $B''$ lying above $\mathfrak{p}$, 
and furthermore for all but finitely many prime $\mathfrak{p} \in B'$, the prime $\mathfrak{q}$ is unique and we have $\kappa(\mathfrak{q}) = \kappa(\mathfrak{p})$. 
\end{lemma}

\begin{proof}
The first conclusion holds by the going-up theorem. 
$B'_{\mathfrak{p}}$ is a discrete valuation ring for all but finitely many $\mathfrak{p}$. 
For these $\mathfrak{p}$ the integral closure of $B'_{\mathfrak{p}}$ in its field of fraction is still 
$B'_{\mathfrak{p}}$. Hence 
such that $\kappa(\mathfrak{q}) = \kappa(\mathfrak{p})$. 
\end{proof}

The advantage of passing to the integral closures is that we have the same automorphism group at the level of fields and at the level of rings. 
Clearly $R$ is normal over $R_i$ and over $\mathcal{O}_{S,K}$. 
Let $G = \Gal(F/K)$ and let $H_i = \Gal(F/\mathrm{Frac}(R_i))$. 
Then $R^G =  \mathcal{O}_{S,K}$ and  $R^{H_i} = R_i$ as $R$ and $R_i$ are integrally closed. 
Let $\mathcal{T}_i$ be the set of left cosets of $H_i$ in $G$. 
For each prime $\mathfrak{q'}$ of $F$, suppose $\mathfrak{p}' = \mathfrak{q'} \cap K$. 
Let $D_{\mathfrak{q}'} = D(\mathfrak{q}'/\mathfrak{p}')$ and $I_{\mathfrak{q}'} = I(\mathfrak{q}'/\mathfrak{p}')$ 
denote the decomposition and inertia groups of $\mathfrak{q}'$.  
We may view $G$ as a group of permutations of the set $\mathcal{T}_i$. 
The following lemma is part of Lemmas 3.1 and 3.2 of \cite{GTZ07}.
This lemma gives a unified approach of treating ramified and unramified primes.

\begin{lemma}\label{action}
There is a group isomorphism $D_{\mathfrak{q}'}/I_{\mathfrak{q}'} \cong \mathrm{Gal}(\kappa(\mathfrak{q}') / \kappa(\mathfrak{p}'))$, and the number of primes $\mathfrak{q}'_i\subset  R_i$ lying over $\mathfrak{p}'$ such that $\kappa(\mathfrak{q}'_i) = \kappa(\mathfrak{p}')$ is the number of common orbits of $D_{\mathfrak{q}'}$ and $I_{\mathfrak{q}'}$ on $\mathcal{T}_i$. \qed
\end{lemma}

\begin{lemma}\label{cor}
If there is no $\mathfrak{q}'_i$ lying over $\mathfrak{p}'$ with residue degree 1, 
then for all $\bar{x} \in X_\mathfrak{p}(\mathcal{O}/\mathfrak{p})$, 
the orbit under $f_{\mathfrak{p}}$ does not hit $\gamma_\mathfrak{p}$ for the first time at the $M$-th iterate. \qed
\end{lemma}

The above corollary applies to any $\mathfrak{p}'$, but we will apply it to $\mathfrak{p}$. 

\begin{corollary}\label{unified}
Notations as in the paragraph before Lemma \ref{action}. Suppose there is no $1\le i \le s$ such that 
$\kappa(\mathfrak{q}'_i) = \kappa(\mathfrak{p}')$. 
Then there is an element $\sigma \in G$ such that $\sigma $ 
has no fixed point on any of $\mathcal{T}_i$. 
\end{corollary}

\begin{proof}
By Lemma \ref{action},  $D_{\mathfrak{q}'}/I_{\mathfrak{q}'} \cong \mathrm{Gal}(\kappa(\mathfrak{q}') / \kappa(\mathfrak{p}'))$ is cyclic. Choose 
$\sigma\in G$ such that 
$\bar{\sigma} \in D_{\mathfrak{q}'}/I_{\mathfrak{q}'}$ is a generator of $D_{\mathfrak{q}'}/I_{\mathfrak{q}'}$.  
Assume by contradiction 
that some $t \in \mathcal{T}_i$ is fixed by $\sigma$ and write $D_{\mathfrak{q}'} = \cup _{j=1}^m I_{\mathfrak{q}'}\sigma^j$ 
where $m = |D_{\mathfrak{q}'}/I_{\mathfrak{q}'}|$.
Now $D_{\mathfrak{q}'}(t) = \cup ^m_{j=1}I_{\mathfrak{q}'}\sigma^j (t) = I_{\mathfrak{q}'} (t)$. 
Hence by Lemma \ref{action} 
this is a contradiction. 
\end{proof}

\begin{example}
The converse of Corollary \ref{unified} is not true. 
Let $K = \mathbb{Q}, s=1,\mathfrak{p}' = (2)$ and let $\mathrm{Frac}(R_1)  = F = \mathbb{Q}(\sqrt{-1})$. Then the generator $\sigma$ of $G$ lies in $I_{\mathfrak{q}'}$, has no fixed point on $\mathcal{T}_1$, and yet 
$\kappa(\mathfrak{q}'_1) = \kappa(\mathfrak{p}')$. 
\end{example}

We replace $K$ by $E$ as in the paragraph before Lemma \ref{fixed} from now on. 
Our goal is to apply the Chebotarev density theorem to obtain a family of primes $\mathfrak{p}'$ satisfying Lemma \ref{fixed}. However, the prime $\mathfrak{p}$ we found might be ramified. So potentially we might need a different way of finding a conjugacy class of a certain Galois group to apply 
the Chebotarev density theorem. 
By Corollary \ref{unified}, there is a unified way to treat the cases when the inertia group is identity or not. 

\begin{proof}[Proof of Theorem \ref{key}]

Use $S$ as in Section \ref{notation}.
Fix a prime $\overline{\mathfrak{p}}$ of $\mathcal{O}_{\bar{K}} $ as in Lemma \ref{fixed}. 
Let $F/K$ be obtained by composing the defining field of all $\beta \in X(K)$ such that the orbit of $\beta$ hits the set $\{\gamma_1,...,\gamma_r\}$ for the first time at the $M$-th iterate. 
Denote by $B$ the set of all such $\beta$. Suppose $B = \{\beta_1,...,\beta_s\}$. 
Suppose $A_i$ and $A$ are the smallest $\mathcal{O}_{S,K}$-algebras over which $\beta_i$ and $B$ are defined. 
By our assumption in Section \ref{notation}, $A_i$ and $A$ are all finitely generated modules over $\mathcal{O}_{S,K}$.  
Then we have inclusions of rings $ \mathcal{O}_{S,K} \subseteq  A_i\subseteq A$. 
Let $R$ be the integral closure of $A$ in its field of fraction. 
Let $R_i$ be the integral closure of $A_i$ in its field of fraction. 
Then the primes of $R_i$ can be identified with primes of $\mathrm{Frac}(R_i)$ whose intersection with $K$ is not in $S$.


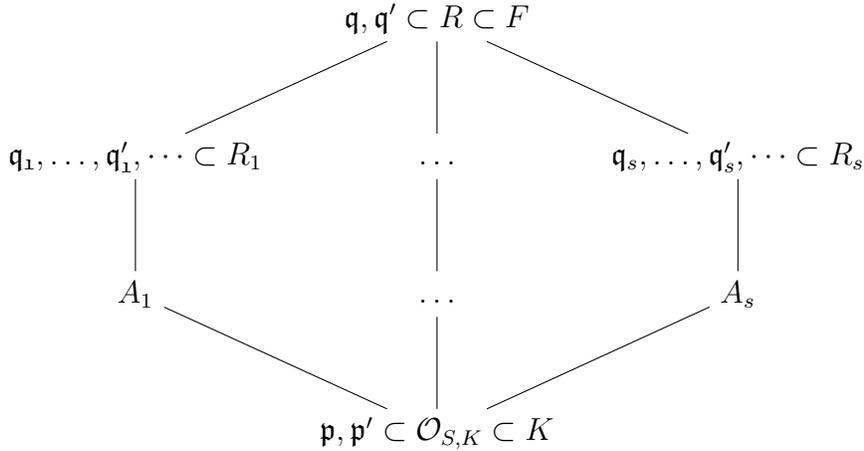
\begin{figure}[h!]\centering
\begin{tikzpicture}[description/.style={fill=white,inner sep=2pt}]
\matrix (m) [matrix of math nodes, row sep=1.5em,
column sep=0.3em, text height=1.5ex, text depth=0.25ex]
{&&&& \mathfrak{q},\mathfrak{q'}  \subset R\subset F &&&& \\
&&&&&&&&\\
\mathfrak{q_1},\dots, \mathfrak{q'_1,}\dots \subset R_1 &&&& \dots &&&& \mathfrak{q}_s, \dots, \mathfrak{q}'_s,\dots \subset R_s\\
 &&&&&&&& \\ 
A_1 &&&& \dots &&&&  A_s\\
&&&&&&&&&&&& \\
&&&&\mathfrak{p}, \mathfrak{p'} \subset \mathcal{O}_{S,K} \subset K &&&&\\};

\path[-] (m-1-5) edge (m-3-1)
		         edge (m-3-9)
                 edge (m-3-5)
		 (m-3-1) edge (m-5-1)
		 (m-3-5) edge (m-5-5)
		 (m-3-9) edge (m-5-9)
		 (m-7-5) edge (m-5-1)
		         edge (m-5-5)
                 edge (m-5-9);
\end{tikzpicture}
\caption{The Diagram}
\end{figure}

By Lemma \ref{fixed}, there is no primes $\mathfrak{q}_i  \in R_i$ lying over $\mathfrak{p}$ such that $\kappa(\mathfrak{q}_i) = \kappa(\mathfrak{p})$. 
Then we find a $\sigma \in G$ as in the proof of Corollary \ref{unified}. 
Look at the nonzero prime ideals of $K$. By the Chebotarev density theorem, there exists a set of primes $\mathcal{P}\subseteq{M_K}$ of positive density, which consists of unramified primes $\mathfrak{p}'$ not in the finite set given by Lemma \ref{integral} such that the Frobenius conjugacy class $\mathrm{Fr}_{\mathfrak{p}'} \in \Gal(F/K)$ coincides with the conjugacy class of $\sigma$. In this case $I_{\mathfrak{p}'} = 1$. By the above paragraph and Lemma \ref{action}, for all such $\mathfrak{p}'$ there is no prime $\mathfrak{q}'_i \subseteq  R_i$ lying over $\mathfrak{p'}$ such that $\kappa(\mathfrak{q'_i}) = \kappa(\mathfrak{p'})$.
It follows that none of $\mathfrak{q'_i}$ has residue degree $1$ over $\kappa(\mathfrak{p'})$. 
By Lemma \ref{cor} there is no ${x}_{\mathfrak{p}} \in X_\mathfrak{p}(\mathcal{O}/\mathfrak{p})$, 
whose orbit under $f_{\mathfrak{p}}$ hits $\gamma_\mathfrak{p}$ for the first time at the $M$-th iterate. 
On the other hand, if ${x}_{\mathfrak{p}}$ hits $(\gamma_i)_\mathfrak{p}$ for the first time at the $m$-th iteration with some $m>M$, then the forward image $f^{m-M}_{\mathfrak{p}}({x}_{\mathfrak{p}})$ hits $(\gamma_i)_\mathfrak{p}$ for the first time at the $M$-th iteration. This is a contradiction as $\mathfrak{p}\in \mathcal{P}$. 

Therefore, for all $m \ge M$ and for all $\mathfrak{p'} \in \mathcal{P}$ there is no $\eta_\mathfrak{p'} \in X_{\mathfrak{p'}}(\kappa(\mathfrak{p'}))$ such that $f_\mathfrak{p'}^m(\eta_\mathfrak{p'}) \in \{(\gamma_1)_\mathfrak{p'},...,(\gamma_r)_\mathfrak{p'}\}$.

\end{proof}


\section{Applying the Approximate $p$-adic Interpolation}

We need the following theorem.

\begin{theorem}[\cite{BGT16}, Theorem 11.11.3.1]\label{Bell}
Let $E$ be a $N$-by-$N$ matrix with entries in $\mathfrak{o}_v$.
Suppose $E^2 = E$. 
If $f \in \mathfrak{o}_v\langle x_1,\dots,x_N\rangle^N$ 
satisfies
$f (x) \equiv Ex \pmod {p^c}$
for some $c > 1/(p - 1)$, then there exists 
$g \in \mathfrak{o}_v\langle x_1,\dots,x_N, z\rangle^N$ 
such that
$\|g(x, n) - f^n (x)\| \le p^{-nc}$
for each $n \in\mathbb{Z}_{\ge 0}$.
\end{theorem}


For each point $a\in \mathcal{X}(\mathfrak{o}_v)$, let $\mathcal{U}_{\bar{a}}$ be the set $\{\beta\in \mathcal{X}(\mathfrak{o}_v): {\beta}_{\mathfrak{p}} = {a}_{\mathfrak{p}}\}$. 
By the argument in \cite{BGT15}, the completed local ring $\hat{\mathcal{O}}_{\bar{a}}$ is isomorphic to a formal
power series ring $\mathfrak{o}_v[[x_1, . . . , x_N]]$. 
For completeness, we include that proof here. Since
$a\in \mathcal{X}$ is smooth, the quotient $\hat{\mathcal{O}}_{\bar{a}}/(\pi)$ is regular. By the Cohen structure theorem for
regular local rings (see \cite{Coh46}, Theorem 9 or \cite{Mat89}, Theorem 29.7), 
the quotient ring $\hat{\mathcal{O}}_{\bar{a}}/(\pi)$ is
isomorphic to a formal power series ring of the form $\kappa_{\mathfrak{p}}[[y_1, . . . , y_N]]$. 
Choosing $x_i \in \hat{m}_v$ for
$i = 1,\dots, N$ such that the residue class of each $x_i$ is equal to $y_i$, we obtain a minimal basis
${\pi, x_1, . . . , x_N}$ for the maximal ideal $\hat{m}_v$ of $\hat{\mathcal{O}}_{\bar{a}}$. (see \cite{Coh46}) 

By the basic theory of ideals, 
there is a one-to-one correspondence between the points in 
$\mathcal{X}(\mathfrak{o}_v)$ that reduces to $\bar{a}$ and the primes $ \mathfrak{q}$ in $\mathcal{O}_{\bar{a}}$ such that $\mathcal{O}_{ \bar{a}}/\mathfrak{q}\cong \mathfrak{o}_v$. 
Passing to the completion, we have a one-to-one correspondence 
between the points in $\mathcal{X}(\mathfrak{o}_v)$ that reduces to $\bar{a}$ and 
the primes $\mathfrak{q}$ of $\hat{\mathcal{O}}_{\mathcal{X}, \bar{a}}$ such that 
$\mathfrak{q}$ of $\hat{\mathcal{O}}_{\mathcal{X}, \bar{a}}/ \mathfrak{q} \cong \mathfrak{o}_v$. 

Replacing $f$ by an iterate and replacing $a$ by a forward image under $f$, we may assume that $f$ maps $\mathcal{U}_{\tilde{a}}$ to itself. 
Then there is a $\mathfrak{p}$-adic analytic isomorphism $\iota: \mathcal{U}_{\tilde{a}} \rightarrow \mathfrak{o}_v^N$, 
and the restriction of $f|_{\mathcal{U}_{\tilde{a}} }$ 
can be conjugated to an analytic endomorphism defined over $\mathfrak{o}_v$. 
Denote by $F  =  (F_1, \dots, F_N)$ this function. 
Moreover, for each $i$, we have
\begin{equation}\label{iterate}
F_i(x_1, \dots, x_N) =  \frac{1}{\pi} H_i(\pi x_1,\dots, \pi x_N)
\end{equation}
for some $H_i \in \mathfrak{o}_v[[x_1,\dots,x_N]]$.
The advantage of conjugating by $(x_1,\dots,x_N)\mapsto (\pi x_1,\dots, \pi x_N)$ has the advantage that the coordinates of the initial point is mapped to the maximal ideal $\mathfrak{p}$.  
The reader might refer to \cite{bgt10}, pp.1658-1659 and \cite{BGT15}, pp.9-11 for more details.

We need the following theorem, which a modification to the analytic case of Theorem 11.11.3.1 of \cite{BGT16}.

\begin{proposition}\label{book}
Let $p\ge 3$ be a prime number, and let $N\ge 2$ be an integer. 
Suppose $Y $ be the $N$-dimensional unit disk $ \mathfrak{o}_v^N$ over $K_\mathfrak{p}$ 
and suppose
$F: Y\rightarrow Y$ is an algebraic endomorphism defined over $K_\mathfrak{p}$.  
Suppose $V'\subseteq Y$ is an analytic subvariety defined over $K_\mathfrak{p}$,  
and let $a\in Y(\mathfrak{o}_v)$ be a point. Assume the following conditions are verified: 

\begin{enumerate}
\item \label{ideal}
$a = (a_1,\dots, a_N)$ and each $a_i\in\mathfrak{p}$;
\item \label{congruence}
the endomorphism $F$ is given by 
$$(x_1,\dots, x_N) \mapsto (F_1(x),\dots, F_N(x)) $$
for analytic functions $F_i \in  \mathfrak{o}_v\langle x_1,\dots,x_n\rangle$ 
and furthermore for each $i = 1,\dots, N$ we have 
$$ F_i(x_1,\dots,x_n)  \equiv \sum_{j =1 }^N a_{ij} x_j\pmod {\mathfrak{p}}$$
with $a_{ij} \in \mathfrak{o}_v$, and the matrix $A: = (a_{ij})$ satisfies that $A^2 = A$;
\item 
\label{approx}
Define $G:\mathfrak{o}_v\rightarrow \mathfrak{o}_v^N$ such that $G(n) = g(a,n)$ where $g$ is the function defined in Theorem \ref{Bell}. There we have $\| G(n) - f^n(a)\| \le p^{-nc}$. 
Assume that there is no $M\in \mathbb{N}$ such that the function $G$ satisfies that $G(n)\in V'$ for all $n\ge M$; note that such a $G$ exists because of Condition \ref{congruence}. 
\item \label{conv}
for any $n\in \mathbb{N}$, the orbit $O_{F^n}(a)$ does not 
converge $\mathfrak{p}$-adically to a periodic point of $F$ lying on $V'$. 
\end{enumerate}
Then the set $S_{V'}:= \{ n\in \mathbb{N}_{\ge 0}:~F^n(a) \in V'(K_\mathfrak{p}) \}$
has the property that 
$$\#\{i\le n:~i\in S_{V'}\} = o(\log^{(m)}(n))$$
for any fixed $m\in \mathbb{N}$ (where $\log^{(m)}$ is the $m$-th iterated logarithmic function). 

\end{proposition}

\begin{remark}
We need condition \ref{approx} to apply Theorem \ref{Bell}. 
It can be satisfied under certain reductions, for more details see the proof of Theorem \ref{main}. 
\end{remark}

\begin{proof}
As in \cite{BGT16}, we may assume that $V'$ is irreducible, $S_{V'}$ is infinite, and the orbit $O_F(a)$ 
does not converge $\mathfrak{p}$-adically to a point in $Y$. 

By Theorem \ref{Bell} the function $G:\mathfrak{o}_v\rightarrow \mathfrak{o}_v^N$ satisfies that for each $n\in \mathbb{N}_{\ge 0}$, we have 
\begin{equation} \label{Jason}
\begin{aligned}
\| F^n (a) - G(n) \| \le p^{-nc}
\end{aligned}
\end{equation}
where $\|(x_1,\dots, x_n)\| = \max\{|x_1|_{\mathfrak{p}},\dots, |x_n|_{\mathfrak{p}} \}$. 
We also have 
\begin{equation} \label{compatible}
F(G(n)) = G(n+1)
\end{equation}
for all $n$. 
There are two cases. 

\textbf{Case 1}. $G$ is not constant.


Look at the sets
$$S_{V',i}:= \{n\in S_{V'}:~n\equiv i \pmod {p^k} \}.$$
By Condition \ref{congruence} there is no $M\in \mathbb{N}$ such that for all $n\ge M$ we have $G(n) \in V’(K_\mathfrak{p})$. 
It follows that there exists an $Q$ vanishing on $V$ such that $L:= Q\circ G:\mathfrak{o}_v\rightarrow \mathfrak{o}_v$ is not identically zero. 
Since $Q\circ F^n(a) = 0$ for all $n\in S_{V‘,i}$, by Equation (\ref{Bell}) we have $|L(n)|_{\mathfrak{p}} \le p^{-nc}$ then. 
By discreteness of zeros of $\mathfrak{p}$-adic nonzero analytic functions we can cover $\mathfrak{o}_v$ by the disks 
$\bar{D}(i,p^{-k})$ with finitely many $i\in \mathfrak{o}_v$ for some $k>0$ such that there is at most one zero of $L$ 
in each of these disks. 

If there is no $\eta\in \bar{D}(i,p^{-k})$ with $L(\eta) = 0$, 
then $S_{V',i}$ is a finite set as (\ref{Bell}) implies that an accumulation point of $S_{V',i}$ will result in a zero of $L$ in $S_{V',i}$. 
If $S_{V',i}$ is an infinite set, 
and we list the elements in $S_{V',i}$ in increasing order as $\{n_j\}_{j\ge 1}$, 
Suppose $L(\eta) = 0$ with $\eta\in S_{V',i}$. 
Let $d$ be the order of vanishing of $L$ at $\eta$. 
Suppose 
$$ L(n) = a_{d} (n-\eta)^d + a_{d+1} (n-\eta)^{d+1}  + \dots$$
Then $|L(n_j)|\le p^{-n_j}$ implies that 
$$|n_j-\eta|_{\mathfrak{p}} \le p^{{-n_j/d} +O(1)}.$$ 
Suppose $S_{V',i} = \{n_1,n_2,\dots\}$. Then we also have
$$|n_{j+1} -\eta|_{\mathfrak{p}} \le p^{{-n_{j+1}/d} +O(1)}.$$ 
Combining them and by the triangle inequality we have 
$$|n_j-n_{j+1}|_{\mathfrak{p}} \le p^{{-n_j/d} +O(1)}.$$ 
Hence there exists $C>1$ 
such that for all sufficiently large $j$, 
we have $n_{j+1} - n_j \ge C^{n_j}$. 

\textbf{Case 2}. $G$ is constant. 

Suppose $G(n)=\beta$ identically. By (\ref{Bell}) we know that $F^n(a)$ converges to $\beta$, and by Equation (\ref{compatible}), $\beta$ is fixed under $F$. This contradicts Condition \ref{conv} above. 
\end{proof}

\begin{proof}[Proof of Theorem \ref{main}]




As in Section \ref{notation}, replacing $f$ by an iterate we may assume that all the periodic points $\gamma_1,\dots,\gamma_r$ are fixed points. 
Let $\mathcal{P}$ be as in Theorem \ref{key}. 
Fix a prime $\mathfrak{p}\in \mathcal{P}$. 
We may assume further that the reduction $a_{\mathfrak{p}}$ is fixed under $f_{\mathfrak{p}}$. 
Fix an analytic isomorphism $\iota: \mathcal{U}_{\hat{a}} \rightarrow Y$ 
where $\mathcal{U}_{\bar{a}}=\{\beta\in \mathcal{X}(\mathfrak{o}_v): \bar{\beta} = \bar{a}\}$ as under the statement of Theorem \ref{Bell} and $Y$ is the $N$-dimensional unit disk $\mathfrak{o}_v^N$. 
Then we obtain a map $F:Y\rightarrow Y$ as in Proposition \ref{book}.

By affine changes of coordinate and replacing $f$ by an iterate, we can assume that Conditions \ref{ideal} and \ref{congruence} of Proposition \ref{book} are verified, as we shall show. 
More precisely, 
suppose 
$$F_i(x) = b_0 + \sum_{j=1}^N b_{j}x_j + \sum_{|n|:= n_1 + \dots + n_N \ge 2} b_{n_1,\dots,n_N} x_1^{n_1}\dots x_N^{n_N}$$
where we omit the subindex $i$. 
Look at the orbit of any $\eta\in\mathfrak{o}_v^N$ under $F$ modulo $\pi^2$: 
$\overline{\eta}, \overline{f(\eta)}, \overline{f^2(\eta)}, \dots$
Replacing $\overline{\eta}$ by a forward image we may that $\overline{\eta}$ is periodic. Replace $F$ by an iterate we may assume that $\overline{\eta}$ is fixed. Choose any $\eta$ such that its residue class modulo $\pi^2$ is $\overline{\eta}$. 
Under these assumptions above, if we do the translation $\sigma_1: x\mapsto x-\eta$ the constant terms of $\sigma_1^{-1}\circ F\circ \sigma_1$ will be divisible by $\pi^2$, and all coefficients of $\sigma_1^{-1}\circ F\circ \sigma_1$ will still belong to $\mathfrak{o}_v$. 
Replacing $F$ by $\sigma_1^{-1}\circ F\circ \sigma_1$ we may assume that the constant terms of $F$ are divisible by $\pi^2$ and all the coefficients of $F$ belong to $\mathfrak{o}_v$.

After that we make a change a variable of the form 
$$\sigma:(a_1,\dots,a_N)\mapsto (\pi a_1,\dots, \pi a_N)$$
then the conjugated map is
$$ \sigma^{-1}\circ F_i\circ\sigma(x) = \frac{b_0}{\pi} + \sum_{j=1}^N b_{j}x_j + \sum_{|n|= n_1 + \dots + n_N \ge 2} \pi^{|n|-1} \cdot b_{n_1,\dots,n_N}  x_1^{n_1}\dots x_N^{n_N}.$$ 
This conjugation makes the terms of $F$ with degree at least $2$ divisible by $\pi$ and yet the constant term is still divisible by $\pi$. 
Now again replacing by $F$ by an iterate we may assume that the matrix $A$ of the linear part of $F$ satisfies $A^2 = A$. This enables us to apply Proposition \ref{book}. 

Furthermore as we have observed, after the transformation
$$(a_1,\dots,a_N)\mapsto (\pi a_1,\dots, \pi a_N)$$ 
our configuration will meet the requirement of Condition \ref{ideal}.  

By Theorem \ref{key} we know that $\mathfrak{p}$-adically any subsequence of $\{f^n(a)\}_{n\ge 0}$ is not convergent to any of the periodic points on $V$. 
It is equivalent to the statement that 
$F^n(\iota(a))$ does not converge to any of the periodic points on $V'$.  
In other words, Condition \ref{conv} of Proposition \ref{book} is satisfied. 
Condition \ref{approx} is satisfied as otherwise the Zariski closure of $\{\iota^{-1}(G(n))~|~n\ge M\}$ would be a positive-dimensional periodic subvariety of $V$. 
Then we can apply Proposition \ref{book} to conclude the desired gap principle.
\end{proof}

\section*{Acknowledgement}

The author would like to express his gratitude to his advisor Thomas Tucker for suggesting this topic and for many useful discussions. 
The author also wants to thank Dragos Ghioca and Joseph Silverman for useful suggestions.

\bibliography{refFile}{}
\bibliographystyle{amsalpha}

\end{document}